\documentclass[11pt,oneside,reqno]{amsart}
\usepackage[pagewise]{lineno}
\usepackage{amscd,amssymb,amsmath,amsbsy,amsthm}
\usepackage{comment,color}
\usepackage{graphicx}
\usepackage{tikz}
\usetikzlibrary{arrows,shapes,positioning}
\usepackage[all]{xy}
\usepackage[colorlinks,plainpages,urlcolor=blue]{hyperref}
\usepackage[width=6.5in,height=8.0in,centering]{geometry}
\usepackage{fancyvrb,bbm}
\usepackage[mathscr]{euscript}
\usepackage{graphicx}
\usepackage{booktabs}

\theoremstyle{definition}

\newtheorem*{rmk}{Remark}

\theoremstyle{plain}
\newtheorem{thm}{Theorem}[section]
\newtheorem{lem}[thm]{Lemma}

\usepackage{amsaddr}

\begin{document}

\title[On the addition of squares of units modulo n]%
{On the addition of squares of units modulo n}

\author[Mohsen Mollahajiaghaei]{Mohsen Mollahajiaghaei}
\address{Department of Mathematics,
University of Western Ontario,\\ London, Ontario, Canada N6A 5B7}
\email{\href{mailto:mmollaha@uwo.ca}{mmollaha@uwo.ca}}



\begin{abstract}
Let $\mathbb{Z}_n$ be the ring of residue classes modulo $n$, and let $\mathbb{Z}_n^{\ast}$ be the group of its units. 90 years ago, Brauer obtained a formula for the number of representations of $c\in \mathbb{Z}_n$ as the sum of $k$ units. Recently, Yang and Tang in [Q. Yang, M. Tang, On the addition of squares of units and nonunits modulo $n$, J. Number Theory., 155 (2015) 1--12] gave a formula for the number of solutions of the equation $x_1^2+x_2^2=c$ with $x_{1},x_{2}\in \mathbb{Z}_n^{\ast}$. In this paper, we generalize this result. We find an explicit formula for the number of solutions of the equation $x^2_{1}+\cdots+x^2_{k}=c$ with $x_{1},\ldots,x_{k}\in \mathbb{Z}_n^{\ast}$.

\end{abstract}
\maketitle
\setcounter{tocdepth}{1}
\section{Introduction}
Let $\mathbb{Z}_n$ be the ring of residue classes modulo $n$, and let $\mathbb{Z}_n^{\ast}$ be the group of its units. Let $c\in \mathbb{Z}_n$, and let $k$ be a positive integer. Brauer in \cite{Brauer} gave a formula for the number of solutions of the equation $x_{1}+\cdots+x_k=c$ with $x_{1},\ldots,x_{k}\in \mathbb{Z}_n^{\ast}$.  In \cite{sander} Sander found the number of representations of a fixed residue class mod $n$ as the sum of two units in
$\mathbb{Z}_{n}$, the sum of two non-units, and the sum of mixed
pairs, respectively. In \cite{Rocky} the results of Sander were generalized into an arbitrary finite commutative ring, as sum of $k$ units and sum of $k$ non-units, with a combinatorial approach.

The problem of finding explicit formulas for the number of representations of 
a natural number $n$ as the sum of $k$ squares is one of the most interesting problems in number theory. For example, if $k=4$, then Jacobi's four-square theorem states that this number is $8 \sum_{m|c}m$ if $c$ is odd and 24 times the sum of the odd divisors of $c$ if $c$ is even. See \cite{toth} and the references given there for historical remarks.

Recently, T\'{o}th \cite{toth} obtained formulas for the number of solutions of the equation 
\begin{equation*}
a_1x_{1}^{2}+\cdots+a_kx_k^2=c,
\end{equation*}
where $c\in \mathbb{Z}_n$, and $x_{i}$ and $a_i$ all belong to $\mathbb{Z}_n$.

Now, consider the equation
\begin{equation}\label{equation}
x_{1}^{2}+\cdots+x_k^2=c,
\end{equation}
where $c\in \mathbb{Z}_n$, and $x_{i}$ are all units in the ring $\mathbb{Z}_n$.
We denote the number of solutions of this equation by $\mathscr{S}_{sq}(\mathbb{Z}_n,c,k)$. In \cite{yang} Yang and Tang obtained a formula for $\mathscr{S}_{sq}(\mathbb{Z}_n,c,2)$. In this paper we provide an explicit formula for $\mathscr{S}_{sq}(\mathbb{Z}_n,c,k)$, for an arbitrary $k$. Our approach is combinatorial with the help of spectral graph theory.
\section{Preliminaries}
In this section we present some graph theoretical notions and properties used in the paper. See, e.g., the book \cite{Godsil}.
Let $G$ be an additive group with identity $0$. For $S \subseteq G$,
the \textit{Cayley graph} $X =Cay(G,S)$ is the directed graph having vertex set $V(X)=G$ and
edge set $E(X) = \{(a,b); b-a\in S\}$. Clearly, if $0\notin S$,
then there is no loop in $X$, and if $0\in S$, then there is
exactly one loop at each vertex. If $-S=\{-s; s\in S \}= S$, then there is an edge from $a$ to $b$ if and only if there is an edge from $b$ to $a$.

Let $\mathbb{Z}_n^{\ast 2}=\{x^{2}; x\in \mathbb{Z}_n^{\ast}\}$. The \textit{quadratic unitary Cayley graph} of $\mathbb{Z}_n$, $G^{2}_{\mathbb{Z}_n} = Cay(\mathbb{Z}_n;\mathbb{Z}_n^{\ast 2})$, is defined as the directed Cayley graph
on the additive group of $\mathbb{Z}_n$ with respect to $\mathbb{Z}_n^{\ast 2}$; that
is, $G^{2}_{\mathbb{Z}_n}$ has vertex set $\mathbb{Z}_n$ such that there is an edge from $x$ to $y$ if and only if $y-x\in \mathbb{Z}_n^{\ast 2}$. Then the out-degree of each vertex is $|\mathbb{Z}_n^{\ast 2}|$.

Let $G$ be a graph, and let $V(G) =\{v_1,\ldots,v_n\}$. The \textit{adjacency matrix} $A_{G}$ of $G$ is defined in a natural way. Thus, the rows and the columns of $A_{G}$ are labeled by $V(G)$. For $i,j$, if there is an edge from $v_i$ to $v_{j}$ then $a_{v_i v_j} = 1$; otherwise $a_{v_i v_j} = 0$. We will write it simply $A$ when no confusion can arise. For the graph $G^{2}_{\mathbb{Z}_n}$ the matrix $A$ is symmetric, provided that -1 is a square mod $n$.

We write $J_{m}$ for the $m\times m$ all 1-matrix. The identity $m\times m$ matrix will be denoted by $I_m$.

The complete graph on $m$ vertices with loop at each vertex is denoted by $K_{m}^{l}$. Thus, the adjacency matrix of $K_{m}^{l}$ is $J_{m}$.

A \textit{walk} in a graph $G$ is a sequence $v_{0},e_{1},v_1,e_2,\ldots,e_n,v_n$ so that $v_{i}\in V(G)$ for every $0\leq i \leq n$, and $e_i$ is an edge from $v_{i-1}$ to $v_i$, for every $1\leq i \leq n$. We denote by $w_{k}(G,i,j)$ the number of walks of length $k$ from $i$ to $j$ in the graph $G$.

One application of the adjacency matrix is to calculate the number of walks between two vertices. 
\begin{lem}\cite[Lemma 8.1.2]{Godsil}\label{walkmatrix}
Let $G$ be a directed graph, and let $k$ be a positive integer. Then the number of walks from vertex $i$ to vertex $j$ of length $k$ is the entry on row $i$ and column $j$ of the matrix $A^k$, where $A$ is the adjacency matrix.
\end{lem}
The next theorem provides the connection between $\mathscr{S}_{sq}(\mathbb{Z}_{p^{\alpha}},c,k)$ and $w_{k}(G^{2}_{\mathbb{Z}_{p^{\alpha}}},0,c)$.
\begin{thm}
Let $p$ be an odd prime number and $\alpha$ be a positive integer. Then 
$$\mathscr{S}_{sq}(\mathbb{Z}_{p^{\alpha}},c,k)=2^kw_{k}(G^{2}_{\mathbb{Z}_{p^{\alpha}}},0,c).$$
\end{thm}
\begin{proof}
Consider the graph $G^{2}_{\mathbb{Z}_{p^{\alpha}}}$. Let $(x_1,\ldots,x_k)\in (\mathbb{Z}_{p^{\alpha}}^{\ast})^k$ such that $x_1^2+x_2^2+\cdots+x_{k}^2=c$. Then $0,x_{1}^2,x_1^2+x_2^2,\ldots,x_1^2+x_2^2+\cdots+x_{k}^2=c$ is a walk of length $k$ from $0$ to $c$.

Now, let $0=a_0,a_1,\ldots,a_k=c$ be a walk of length $k$. Then $a_{i}-a_{i-1}=y_{i}^{2}$, where $y_{i}\in \mathbb{Z}_{p^{\alpha}}^{\ast}$ for $i=1,\ldots,k$. Hence $y_1^2+y_2^2+\cdots+y_{k}^2=c$. Then the set $\{(\epsilon_{k}y_{1},\ldots,\epsilon_{k} y_{k}); \epsilon_{i}\in \{1,-1\}\}$ is a set of solutions of size $2^k$, which proves the theorem.
\end{proof}
The \textit{tensor product} $G_{1}\otimes G_{2}$ of two graphs
$G_{1}$ and $G_{2}$ is the graph with vertex set $V(G_{1}\otimes
G_{2}):=V(G_{1})\times V(G_{2})$, with edges specified by putting
$(u,v)$ adjacent to $(u',v')$ if and only if $u$ is adjacent to
$u'$ in $G_{1}$ and $v$ is adjacent to $v'$ in $G_{2}$. It can be easily verified that the number of edges in $G_{1}\otimes G_{2}$ is equal to the product of the number of edges in the graphs $G$ and $H$.
\begin{lem}\label{adjtensor}
The adjacency matrix of $G \otimes H$ is the tensor product of the adjacency matrices of $G$ and $H$.
\end{lem}
The rest of paper is organized as follows. In section 3 we reduce the case $\mathscr{S}_{sq}(\mathbb{Z}_{n},c,k)$ to the cases $\mathscr{S}_{sq}(\mathbb{Z}_{p},c,k)$ and $\mathscr{S}_{sq}(\mathbb{Z}_{2^{\alpha}},c,k)$. We show that if $p$ is an odd prime number, then $G^{2}_{\mathbb{Z}_{p^{\alpha}}}\cong G^{2}_{\mathbb{Z}_p}\otimes K^{l}_{p^{\alpha-1}}$. Section 4 is devoted to the study of $\mathscr{S}_{sq}(\mathbb{Z}_{p},c,k)$, where $p\equiv 1 \mod 4$. In this section, we write $A^k$ as a linear combination of matrices $A$, $J_{p}$ and $I_p$, and then we obtain a formula for $\mathscr{S}_{sq}(\mathbb{Z}_{p^{\alpha}},c,k)$. Similarly, we find a formula for $\mathscr{S}_{sq}(\mathbb{Z}_{p^{\alpha}},c,k)$, where $p\equiv 3 \mod 4$, in section 5. Last section, provides an explicit formula for $\mathscr{S}_{sq}(\mathbb{Z}_{2^{\alpha}},c,k)$ by direct counting.

\section{General results}
In this section, we reduce the case $\mathscr{S}_{sq}(\mathbb{Z}_{n},c,k)$ to the cases $\mathscr{S}_{sq}(\mathbb{Z}_{p},c,k)$ and $\mathscr{S}_{sq}(\mathbb{Z}_{2^{\alpha}},c,k)$.

The next lemma shows that the function $n \rightarrow \mathscr{S}_{sq}(\mathbb{Z}_n,c,k)$ is multiplicative.
\begin{lem}\label{product}
Let $m,n$ be coprime numbers. Then $\mathscr{S}_{sq}(\mathbb{Z}_{mn},c,k)=\mathscr{S}_{sq}(\mathbb{Z}_{m},c,k)\cdot\mathscr{S}_{sq}(\mathbb{Z}_{n},c,k)$.
\end{lem}
\begin{proof}
The proof follows using the Chinese remainder theorem.
\end{proof}
\begin{lem}\label{squarequotient}
Let $p$ be an odd prime number, and let $m$ be the ideal generated by $p$ in the ring $\mathbb{Z}_{p^{\alpha}}$. Let $u\in \mathbb{Z}_{p^{\alpha}}^{\ast 2}$ and $r\in m$. Then $u+r\in \mathbb{Z}_{p^{\alpha}}^{\ast 2}$.
\end{lem}
\begin{proof}
For this to happen, it is enough to show that $1+r$ belongs to $\mathbb{Z}_{p^{\alpha}}^{\ast 2}$. We know that $r$ is a nilpotent element of $\mathbb{Z}_{p^{\alpha}}$. Let $\lambda$ be a sufficiently large integer. Then $(1+r)^{p^{\lambda}}=1$. Hence, $(1+r)^{p^{\lambda}+1}=1+r$.
\end{proof}
\begin{thm}\label{tensorprodjacob}
Let $p$ be an odd prime number, and let $\alpha$ be a positive integer. Then $G^{2}_{\mathbb{Z}_{p^{\alpha}}}\cong G^{2}_{\mathbb{Z}_p}\otimes K^{l}_{p^{\alpha-1}}$.
\end{thm}
\begin{proof}
Let $m$ be the ideal generated by $p$, and $\mathbb{Z}_{p^{\alpha}}=\bigcup_{i=1}^{p}(m+r_{i})$, where $m+r_{i}$ is a coset of the maximal ideal $m$ in $\mathbb{Z}_{p^{\alpha}}$. The ring $\mathbb{Z}_{p^{\alpha}}/m$ is isomorphic to the field $\mathbb{Z}_{p}$. Then for each $r\in \mathbb{Z}_{p^{\alpha}}$ there is a unique $i$ and $n_r\in m$ such that $r=r_i+n_r$. Let $\psi:G^{2}_{\mathbb{Z}_{p^{\alpha}}} \longrightarrow G^{2}_{\mathbb{Z}_{p}}\otimes K^{l}_{p^{\alpha-1}}$ be defined by $\psi(r):=(r_i+m,n_r)$. Obviously, this map is a bijection. Now, let $(r,r')$ be a directed edge in $G^{2}_{\mathbb{Z}_{p^{\alpha}}}$. We show that $(\psi(r),\psi(r'))$ is also a directed edge in $G^{2}_{\mathbb{Z}_{p}}\otimes K^{l}_{p^{\alpha-1}}$. By definition, $\psi(r)=(r_{i}+m,n_{r})$ and $\psi(r')=(r_{j}+m,n_{r'})$. We have $r'-r\in \mathbb{Z}_{p^{\alpha}}^{\ast 2}$. Thus, $r_{j}-r_{i}+n_{r'}-n_{r}\in \mathbb{Z}_{p^{\alpha}}^{\ast 2}$. Hence by Lemma \ref{squarequotient}, $r_{j}-r_{i}\in \mathbb{Z}_{p^{\alpha}}^{\ast 2}$. Then $r_{j}-r_{i}+m\in (\mathbb{Z}_{p^{\alpha}}/ m)^{\ast 2}$. Since the number of edges of $G^{2}_{\mathbb{Z}_{p^{\alpha}}}$ and $G^{2}_{\mathbb{Z}_p}\otimes K^{l}_{p^{\alpha-1}}$ are the same, the proof is complete.
\end{proof}
By the aforementioned theorem, we see
\begin{equation*}
\begin{split}
A^{k}_{G^{2}_{\mathbb{Z}_{p^{\alpha}}}}&=
A_{G^{2}_{\mathbb{Z}_p} }^{k}\otimes A_{K^{l}_{p^{\alpha-1}}}^{k}\\
&=A^{k}_{G^{2}_{\mathbb{Z}_p}}\otimes J_{p^{\alpha-1}}^{k}.\\
\end{split}
\end{equation*}
\section{\texorpdfstring{$\mathscr{S}_{sq}(\mathbb{Z}_{p^{\alpha}},c,k)$}{p=1mod4} where \texorpdfstring{$p\equiv 1 \mod 4$}{p=1mod4}}
In this section, we find $\mathscr{S}_{sq}(\mathbb{Z}_{p^{\alpha}},c,k)$, where $p$ is a prime number with $p\equiv 1 \mod 4$.

An {\it strongly regular graph} with parameters $(n,k,
\lambda,\mu)$ is a simple graph with $n$ vertices that is regular of
valency $k$ and has the following properties:

$\bullet$ For any two adjacent vertices $x,y$, there are exactly
$\lambda$ vertices adjacent to both $x$ and $y$.

$\bullet$ For any two non-adjacent vertices $x,y$, there are
exactly $\mu$ vertices adjacent to both $x$ and $y$.

Let $p$ be a fixed prime number with $p \equiv 1 \mod 4$. The Paley graph $P_{p}$
is defined by taking the field $\mathbb{Z}_p$ as vertex set, with two vertices $x$ and $y$ joined by an edge if and only if $x-y$ is a nonzero square in $\mathbb{Z}_p$.

As in well known (see e.g., \cite[Page 221]{Godsil}), the Paley graph is strongly regular with parameters $(p,\frac{p-1}{2},\frac{p-5}{4},\frac{p-1}{4})$. The fact that Paley graph is strongly regular shows that $A^2$ can be written as a linear combination of matrices $A$, $J_p$ and $I_p$.

\begin{lem}\cite[Page 219]{Godsil}
Let $p$ be a prime number such that $p \equiv 1\mod 4$. Then the adjacency matrix of the Paley graph $P_{p}$ satisfies
\begin{equation}\label{main}
A_{P_{p}}^2=-A_{P_{p}}+(\frac{p-1}{4})J_p+(\frac{p-1}{4})I_p.
\end{equation}
\end{lem}
Although the graph $G_{\mathbb{Z}_p}^{2}$ is a directed graph and $P_p$ is a simple graph, they share the same adjacency matrix. Then $A_{G_{\mathbb{Z}_p}^{2}}^{n}$ can be written as a linear combination of $A_{G_{\mathbb{Z}_p}^{2}}$, $I_{p}$ and $J_{p}$.\\
Let
\begin{equation}
A^{n+1}=a_{n,p}A+b_{n,p}J_p+c_{n,p}I_p.
\end{equation}
Then
\begin{equation*}
A^{n+2}=a_{n,p}A^2+\frac{p-1}{2}b_{n,p}J_p+c_{n,p}A.
\end{equation*}
Now, by equation (\ref{main}), we have
\begin{equation*}
A^{n+2}=(a_{n,p}a_{1,p}+c_{n,p})A+(\frac{p-1}{2}b_{n,p}+a_{n,p}b_{1,p})J_p+(a_{n,p}c_{1,p})I_p.
\end{equation*}
Then we see that
\begin{displaymath}
\left\{ \begin{array}{ll}
a_{n+1,p}=a_{n,p}a_{1,p}+c_{n,p},& a_{1,p}=-1, a_{2,p}=\frac{p+3}{4};\\
b_{n+1,p}=\frac{p-1}{2}b_{n,p}+a_{n,p}b_{1,p},&b_{1,p}=\frac{p-1}{4}, b_{2,p}=(\frac{p-1}{4})(\frac{p-3}{2});\\
c_{n+1,p}=a_{n,p}c_{1,p}, &c_{1,p}=\frac{p-1}{4}, c_{2,p}=-\frac{p-1}{4}.\
\end{array} \right.
\end{displaymath}
From the first and last equations, we have the following homogeneous linear recurrence relation
$$a_{n,p}=\frac{p-1}{4}a_{n-2,p}-a_{n-1,p}.$$
Since $a_1=-1$ and $a_2=\frac{p+3}{4}$, we deduce
\begin{equation*}
a_{n,p}=(\dfrac{\sqrt{p}-1}{2\sqrt{p}}) (\dfrac{-1+\sqrt{p}}{2})^n+(\dfrac{\sqrt{p}+1}{2\sqrt{p}}) (\dfrac{-1-\sqrt{p}}{2})^n.
\end{equation*}
Then
\begin{equation}\label{i}\tag{i}
a_{n,p}=\Big( \dfrac{1}{2^{n+1}\sqrt{p}} \Big)\Big( (-1+\sqrt{p})^{n+1} + (-1)^{n} (1+\sqrt{p})^{n+1} \Big).
\end{equation}
Now, we have
\begin{equation*}\label{ii}\tag{ii}
c_{n,p}=\Big( \dfrac{p-1}{2^{n+2}\sqrt{p}} \Big)\Big( (-1+\sqrt{p})^{n} + (-1)^{n-1} (1+\sqrt{p})^{n} \Big).
\end{equation*}
Thus, for $b_{n,p}$ we have the following non-homogeneous linear recurrence relation
\begin{equation*}
b_{n,p}=\dfrac{p-1}{2}b_{n-1,p}+\Big( \dfrac{p-1}{2^{n+1}\sqrt{p}} \Big)\Big( (-1+\sqrt{p})^{n-1} + (-1)^{n-2} (1+\sqrt{p})^{n-1} \Big).
\end{equation*}
Then
\begin{equation*}
b_{n,p}=\beta (\dfrac{p-1}{2})^n+
\Big( \dfrac{p-1}{2^{n+1}\sqrt{p}} \Big)\Big( \dfrac{\sqrt{p}+1}{\sqrt{p}-1}(-1+\sqrt{p})^{n-1} + (-1)^{n-2}\dfrac{\sqrt{p}-1}{\sqrt{p}+1} (1+\sqrt{p})^{n-1} \Big).
\end{equation*}
Since $b_1=\frac{p-1}{4}$, it follows that
\begin{equation*}\label{iii}\tag{iii}
b_{n,p}=\dfrac{(p-5)}{2(p-1)}(\dfrac{p-1}{2})^n+
\Big( \dfrac{p-1}{2^{n+1}\sqrt{p}} \Big)\Big( \dfrac{\sqrt{p}+1}{\sqrt{p}-1}(-1+\sqrt{p})^{n-1} + (-1)^{n-2}\dfrac{\sqrt{p}-1}{\sqrt{p}+1} (1+\sqrt{p})^{n-1} \Big).
\end{equation*}
We can now find $\mathscr{S}_{sq}(\mathbb{Z}_{p},c,k)$.
\begin{equation}\label{field4k1}
\mathscr{S}_{sq}(\mathbb{Z}_{p},c,k)=\left\{ \begin{array}{ll}
2^k (b_{k-1,p}+c_{k-1,p}),& \textrm{if $c = 0$};\\
2^k (a_{k-1,p}+b_{k-1,p}),& \textrm{if $c=x^2$, for some $x\in \mathbb{Z}_p^{\ast}$};\\
2^k b_{k-1,p},&\textrm{otherwise}.\
\end{array} \right.
\end{equation}
The last theorem of this section provides a formula for $\mathscr{S}_{sq}(\mathbb{Z}_{p^{\alpha}},c,k)$.
\begin{thm}\label{4k1}
Let $p$ be a prime number such that $p\equiv 1 \mod 4$. Let $\alpha$ be a positive integer. Then \begin{displaymath}
\mathscr{S}_{sq}(\mathbb{Z}_{p^{\alpha}},c,k)=\left\{ \begin{array}{ll}
p^{(\alpha -1)(k-1)}2^k (b_{k-1,p}+c_{k-1,p}),& \textrm{if $c \equiv 0 \mod p$};\\
p^{(\alpha -1)(k-1)}2^k (a_{k-1,p}+b_{k-1,p}),&\textrm{if $c=x^2$, for some $x\in \mathbb{Z}_{p^{\alpha}}^{\ast}$};\\
p^{(\alpha -1)(k-1)} 2^k b_{k-1,p},&\textrm{otherwise},\
\end{array} \right.
\end{displaymath}
where $a_{k-1,p}$, $c_{k-1,p}$ and $b_{k-1,p}$ are defined by equations (\ref{i}), (\ref{ii}) and (\ref{iii}), respectively, (putting $n=k-1$).
\end{thm}
\begin{proof}
By Theorem \ref{tensorprodjacob} and Lemma \ref{adjtensor}, $A_{G^{2}_{\mathbb{Z}_{p^{\alpha}}}}= A_{G^{2}_{\mathbb{Z}_p}}\otimes A_{K^{l}_{p^{\alpha-1}}}$. Then  
\begin{equation*}
\begin{split}
A^{k}_{G^{2}_{\mathbb{Z}_{p^{\alpha}}}}&= A^{k}_{G^{2}_{\mathbb{Z}_p}}\otimes J_{p^{\alpha-1}}^{k}\\
&=A^{k}_{G^{2}_{\mathbb{Z}_p}}\otimes p^{(\alpha-1)(k-1)}J_{p^{\alpha-1}}.\\
\end{split}
\end{equation*}
Then equation (\ref{field4k1}) and Lemma \ref{walkmatrix}, complete the proof.
\end{proof}
\section{\texorpdfstring{$\mathscr{S}_{sq}(\mathbb{Z}_{p^{\alpha}},c,k)$}{p=3mod4} where \texorpdfstring{$p\equiv 3 \mod 4$}{p=3mod4}}
In this section, we find $\mathscr{S}_{sq}(\mathbb{Z}_{p^{\alpha}},c,k)$, where $p$ is a prime number with $p\equiv 3 \mod 4$. The main idea is similar to that used in the previous section. We try to write $A_{G_{\mathbb{Z}_{p}}^{2}}^2$ as a linear combination of matrices $A_{G_{\mathbb{Z}_{p}}^{2}}$, $I_{p}$ and $J_p$.

The field $\mathbb{Z}_p$, has no square root of -1. Then for each pair of $(x,y)$ of distinct elements of $\mathbb{Z}_p$, either $x-y$ or $y-x$, but not both, is a square of a nonzero element.
Hence in the graph $G_{\mathbb{Z}_{p}}^{2}$, each pair of distinct vertices is linked by an arc in one and only one direction. Therefore, $A_{G_{\mathbb{Z}_{p}}^{2}}+A_{G_{\mathbb{Z}_{p}}^{2}}^T=J_p-I_p$. The entry on row $a$ and column $b$ of the matrix $A_{G_{\mathbb{Z}_{p}}^{2}}^2$ equals to the size of the set $(a+\mathbb{Z}_{p}^{\ast 2})\cap (b-\mathbb{Z}_{p}^{\ast 2})$. The goal of following lemmas is to find $|(a+\mathbb{Z}_{p}^{\ast 2})\cap (b-\mathbb{Z}_{p}^{\ast 2})|$.
\begin{lem}\label{jam}
Let $a$ and $b$ be elements of $\mathbb{Z}_p$. Then $|(a+\mathbb{Z}_{p}^{\ast 2})\cap (b-\mathbb{Z}_{p}^{\ast 2})|=|(a-b+\mathbb{Z}_{p}^{\ast 2})\cap -\mathbb{Z}_{p}^{\ast 2}|$.
\end{lem}
\begin{proof}
Let $\psi: (a+\mathbb{Z}_{p}^{\ast 2})\cap (b-\mathbb{Z}_{p}^{\ast 2}) \longrightarrow (a-b+\mathbb{Z}_{p}^{\ast 2})\cap -\mathbb{Z}_{p}^{\ast 2}$ be defined by $\psi(r)=r-b$. Obviously, $\psi$ is well-defined and injective. Now, let $c\in (a-b+\mathbb{Z}_{p}^{\ast 2})\cap -\mathbb{Z}_{p}^{\ast 2}$, so there exists $s\in \mathbb{Z}_{p}^{\ast 2}$ such that $c=a-b+s$. Then $\psi(c+b)=c$, which completes the proof.
\end{proof}
\begin{lem}\label{zarb}
Let $a$ be a non-zero element of $\mathbb{Z}_p$. Then $|(a^2+\mathbb{Z}_{p}^{\ast 2})\cap -\mathbb{Z}_{p}^{\ast 2}|=|(1+\mathbb{Z}_{p}^{\ast 2})\cap -\mathbb{Z}_{p}^{\ast 2}|$ and $|(-a^2+\mathbb{Z}_{p}^{\ast 2})\cap -\mathbb{Z}_{p}^{\ast 2}|=|(-1+\mathbb{Z}_{p}^{\ast 2})\cap -\mathbb{Z}_{p}^{\ast 2}|$.
\end{lem}
\begin{proof}
Let $\psi: (a^2+\mathbb{Z}_{p}^{\ast 2})\cap -\mathbb{Z}_{p}^{\ast 2} \longrightarrow (1+\mathbb{Z}_{p}^{\ast 2})\cap -\mathbb{Z}_{p}^{\ast 2}$ be defined by $\psi(r)=ra^{-2}$. Obviously, $\psi$ is well-defined and injective. Now, let $c\in (1+\mathbb{Z}_{p}^{\ast 2})\cap -\mathbb{Z}_{p}^{\ast 2}$. Thus, there exists $s\in \mathbb{Z}_p^{\ast}$ such that $c=1+s^2$. Then $\psi(ca^{2})=c$, which completes the proof.

The proof for the second part is similar.
\end{proof}
Then by Lemmas \ref{jam} and \ref{zarb}, one can easily see that $A^2$ is a linear combination of matrices $A$, $J_p$ and $I_p$.
\begin{lem}\label{lambda}
$|(1+\mathbb{Z}_{p}^{\ast 2})\cap (-\mathbb{Z}_{p}^{\ast 2})|=\frac{p+1}{4}$.
\end{lem}
\begin{proof}
We know that $\Big((1+\mathbb{Z}_{p}^{\ast 2})\cap (-\mathbb{Z}_{p}^{\ast 2})\Big)\cup \Big((1+\mathbb{Z}_{p}^{\ast 2})\cap
(\mathbb{Z}_{p}^{\ast 2})\Big)=1+\mathbb{Z}_{p}^{\ast 2}$, and $\Big((1+\mathbb{Z}_{p}^{\ast 2})\cap (-\mathbb{Z}_{p}^{\ast 2})\Big)\cap \Big((1+\mathbb{Z}_{p}^{\ast 2})\cap
(\mathbb{Z}_{p}^{\ast 2})\Big)=\emptyset$. Then $|(1+\mathbb{Z}_{p}^{\ast 2})\cap
(-\mathbb{Z}_{p}^{\ast 2})|=\frac{p-1}{2}-|(1+\mathbb{Z}_{p}^{\ast 2})\cap (\mathbb{Z}_{p}^{\ast 2})|$. Now, $a\in (1+\mathbb{Z}_{p}^{\ast 2})\cap (\mathbb{Z}_{p}^{\ast 2})$
if and only there exist $b,c\in \mathbb{Z}_p^{\ast}$ such that
$a=1+b^2=c^2$. Thus, $(c-b)(c+b)=1$. Hence $c=\frac{u+u^{-1}}{2}$ and $b=\frac{u-u^{-1}}{2}$,
for $u\in \mathbb{Z}_p^{\ast}-\{1,-1\}$. Then $(1+\mathbb{Z}_{p}^{\ast 2})\cap (\mathbb{Z}_{p}^{\ast 2})=\{ (\frac{u+u^{-1}}{2})^2 ; u\in \mathbb{Z}_p^{\ast}\}-\{1\}$.

If $(\frac{u+u^{-1}}{2})^2=(\frac{v+v^{-1}}{2})^2$, then we have two cases:
\begin{itemize}
\item[(i)] $\frac{u+u^{-1}}{2}=\frac{v+v^{-1}}{2}$. A trivial verification shows that $u=v$ or $u=v^{-1}$.
\item[(ii)] $\frac{u+u^{-1}}{2}=-\frac{v+v^{-1}}{2}$. Then $u=-v$ or $u=-v^{-1}$.
\end{itemize}
Then $|(1+\mathbb{Z}_{p}^{\ast 2})\cap (\mathbb{Z}_{p}^{\ast 2})|=\frac{p-1-2}{4}$, and the lemma follows.
\end{proof}
The following lemma may be proved in much the same way as Lemma \ref{lambda}.
\begin{lem}\label{mu}
$|(-1+\mathbb{Z}_{p}^{\ast 2})\cap (-\mathbb{Z}_{p}^{\ast 2})|=\frac{p-3}{4}$.
\end{lem}
\begin{lem}\label{adjp3mod4}
Let $p$ be a prime number with $p\equiv 3 \mod 4$. Let $A$ be the adjacency matrix of the graph $G_{\mathbb{Z}_{p}}^{2}$. Then
\begin{equation} \label{Asqp3mod4}
A^2=-A+(\frac{p+1}{4})J_p-(\frac{p+1}{4})I_p.
\end{equation}
\end{lem}
\begin{proof}
Let $a,b \in\mathbb{Z}_{p}$.  By Lemma \ref{jam}, $$(A)_{ab}=|(a+\mathbb{Z}_{p}^{\ast 2})\cap (b-\mathbb{Z}_{p}^{\ast 2})|=|(a-b+\mathbb{Z}_{p}^{\ast 2})\cap (-\mathbb{Z}_{p}^{\ast 2})|.$$

If there is an edge from $a$ to $b$, then by Lemmas \ref{zarb} and \ref{mu},
$$(A)_{ab}=|(-1+\mathbb{Z}_{p}^{\ast 2})\cap (-\mathbb{Z}_{p}^{\ast 2})|=\frac{p-3}{4}.$$
If $a\neq b$ and there is no edge from $a$ to $b$, then by a similar argument, we have $(A)_{ab}=\frac{p+1}{4}$. If $a=b$, then by Lemma \ref{jam}, $$(A)_{ab}=|(a+\mathbb{Z}_{p}^{\ast 2})\cap (b-\mathbb{Z}_{p}^{\ast 2})|=|(\mathbb{Z}_{p}^{\ast 2})\cap (-\mathbb{Z}_{p}^{\ast 2})|=0,$$
which establishes equation (\ref{Asqp3mod4}).
\end{proof}
Let \begin{equation*}
A^{n+1}=a_{n,p}A+b_{n,p}J_p+c_{n,p}I_p.
\end{equation*}
Hence
\begin{equation*}
A^{n+1}=a_{n,p}A^2+b_{n,p}\frac{p-1}{2}J_p+c_{n,p}A.
\end{equation*}
Then
\begin{equation*}
A^{n+1}=(c_{n+1,p}-a_{n,p})A+(a_{n,p}\frac{p+1}{4}+b_{n+1,p}\frac{p-1}{2})J_p+(-a_{n,p} \frac{p+1}{4})I_p.
\end{equation*}
Thus, we have
\begin{displaymath}
\left\{ \begin{array}{ll}
a_{n+1,p}=c_{n,p}-a_{n,p},& a_{1,p}=-1, a_{2,p}=\frac{3-p}{4};\\
b_{n+1,p}=\frac{p-1}{2}b_{n,p}+a_{n,p}\frac{p+1}{4},&b_{1,p}=\frac{p+1}{4}, b_{2,p}=\frac{p+1}{4}(\frac{p-1}{2}-1);\\
c_{n+1,p}=-a_{n,p}\frac{p+1}{4}, &c_{1,p}=-\frac{p+1}{4}, c_{2,p}=\frac{p+1}{4}.\
\end{array} \right.
\end{displaymath}
From the first and last equations, we have the following homogeneous linear recurrence relation
$$a_{n+1,p}+a_{n,p}+\frac{p+1}{4}a_{n-1,p}=0.$$
Since $a_{1,p}=-1$ and $a_{2,p}=\frac{3-p}{4}$, we deduce
\begin{equation*}\label{i'}\tag{i'}
a_{n,p}=(\frac{\sqrt{p}+i}{2\sqrt{p}})(\dfrac{-1+i\sqrt{p}}{2})^n+(\frac{\sqrt{p}-i}{2\sqrt{p}})(\dfrac{-1-i\sqrt{p}}{2})^n,
\end{equation*}
where $i=\sqrt{-1}$. Then
\begin{equation*}\label{ii'}\tag{ii'}
c_{n,p}=-\frac{p+1}{4}\Big((\frac{\sqrt{p}+i}{2\sqrt{p}})(\dfrac{-1+i\sqrt{p}}{2})^{n-1}+(\frac{\sqrt{p}-i}{2\sqrt{p}})(\dfrac{-1-i\sqrt{p}}{2})^{n-1}\Big).
\end{equation*}
Thus, for $b_{n,p}$ we have the following non-homogeneous linear recurrence relation
\begin{equation*}
b_{n,p}=\dfrac{p-1}{2}b_{n-1,p}+\frac{p+1}{4}\Big((\frac{\sqrt{p}+i}{2\sqrt{p}})(\dfrac{-1+i\sqrt{p}}{2})^{n-1}+(\frac{\sqrt{p}-i}{2\sqrt{p}})(\dfrac{-1-i\sqrt{p}}{2})^{n-1}\Big).
\end{equation*}
Then
\begin{equation*}
b_{n,p}=\alpha(\frac{p-1}{2})^n+\frac{p+1}{8\sqrt{p}}\Big((\frac{(\sqrt{p}+i)(i\sqrt{p}-1)}{i\sqrt{p}-p})(\dfrac{-1+i\sqrt{p}}{2})^{n-1}+(\frac{(\sqrt{p}-i)(i\sqrt{p}+1)}{i\sqrt{p}+p})(\dfrac{-1-i\sqrt{p}}{2})^{n-1}\Big).
\end{equation*}
Since $b_{1,p}=\frac{p+1}{4}$, it follows that
\begin{equation*}\label{iii'}\tag{iii'}
b_{n,p}=\frac{p-1}{2p}(\frac{p-1}{2})^n+\frac{p+1}{8\sqrt{p}}\Big((\frac{(\sqrt{p}+i)(i\sqrt{p}-1)}{i\sqrt{p}-p})(\dfrac{-1+i\sqrt{p}}{2})^{n-1}+(\frac{(\sqrt{p}-i)(i\sqrt{p}+1)}{i\sqrt{p}+p})(\dfrac{-1-i\sqrt{p}}{2})^{n-1}\Big).
\end{equation*}
Then the number of solutions of the equation (\ref{equation}) is
\begin{displaymath}
\mathscr{S}_{sq}(\mathbb{Z}_{p},c,k)=\left\{ \begin{array}{ll}
2^k (b_{k-1,p}+c_{k-1,p}),& \textrm{if $c=0$};\\
2^k (a_{k-1.p}+b_{k-1,p}),&\textrm{if $c=x^2$, for some $x\in \mathbb{Z}_p^{\ast}$};\\
2^k b_{k-1,p},&\textrm{otherwise}.\
\end{array} \right.
\end{displaymath}
\begin{thm}\label{np4k3}
Let $p$ be a prime number such that $p\equiv 3 \mod 4$. Let $\alpha$ be a positive integer. Then \begin{displaymath}
\mathscr{S}_{sq}(\mathbb{Z}_{p^{\alpha}},c,k)=\left\{ \begin{array}{ll}
p^{(\alpha -1)(k-1)}2^k (b_{k-1,p}+c_{k-1,p}),& \textrm{if $c\equiv 0 \mod p$};\\
p^{(\alpha -1)(k-1)}2^k (a_{k-1,p}+b_{k-1,p}),&\textrm{if $c=x^2$, for some $x\in \mathbb{Z}_{p^{\alpha}}^{\ast}$};\\
p^{(\alpha -1)(k-1)} 2^k b_{k-1,p},&\textrm{otherwise},\
\end{array} \right.
\end{displaymath}
where $a_{k-1,p}$, $c_{k-1,p}$ and $b_{k-1,p}$ are defined by equations (\ref{i'}), (\ref{ii'}) and (\ref{iii'}), respectively, (putting $n=k-1$).
\end{thm}
\begin{proof}
The proof is similar to that of Theorem \ref{4k1}.
\end{proof}
\section{\texorpdfstring{$\mathscr{S}_{sq}(\mathbb{Z}_{2^{\alpha}},c,k)$}{2t}}
In this section we find $\mathscr{S}_{sq}(\mathbb{Z}_{2^{\alpha}},c,k)$. For $\alpha=1$ and $\alpha=2$, this number is easy to find.
\begin{lem}\label{8k}
Let $n=2^\alpha$ such that $\alpha>2$. Then $\mathbb{Z}_{n}^{\ast 2}=\Big\{8k+1;k\in \{0,\ldots,\frac{n}{8}-1\}\Big\}$.
\end{lem}
\begin{proof}
Obviously, $\Big\{8k+1;k\in \{0,\ldots, \frac{n}{8}-1\}\Big\} \supseteq \mathbb{Z}_{n}^{\ast 2}$. It suffices to show that the set $\mathbb{Z}_{n}^{\ast 2}$ has exactly $n/8$ elements. Define the equivalence relation between odd elements of $\mathbb{Z}_{n}$ as follows.
We say $a \sim b$ if and only if $a^2\equiv b^2 \mod 2^\alpha$. It is easy to check that each equivalence class has exactly 4 elements. Hence the number of equivalence classes is $n/8$, which equals to the size of $\mathbb{Z}_{n}^{\ast 2}$.
\end{proof}
Now, we are able to find $\mathscr{S}_{sq}(\mathbb{Z}_{2^{\alpha}},c,k)$.
\begin{thm}\label{n2alpha}
Let $n=2^\alpha$. Then
\begin{displaymath}
\mathscr{S}_{sq}(\mathbb{Z}_{2^{\alpha}},c,k)= \left\{ \begin{array}{ll}
1,& \textrm{if $\alpha=1$ and $c \equiv k \mod 2$};\\
2^{k},& \textrm{if $\alpha=2$ and $c \equiv k \mod 4$};\\
2^{2k+(\alpha -3)(k-1)},& \textrm{if $\alpha>2$ and $c \equiv k \mod 8$};\\
0&\textrm{otherwise}.\
\end{array} \right.
\end{displaymath}
\end{thm}
\begin{proof}
Let $\alpha>2$. Let $A=\{ (y_1,\ldots,y_k); 8\sum_{i=1}^{k} y_i=c-k \}$ and $B= \{ (x_1,\ldots,x_k); \sum_{i=1}^{k} x_i^2=c \}$. Then by Lemma \ref{8k}, there exists a  $4^k$ to $1$ and onto map from $B$ to $A$.  It in easy to see that if $c \equiv k \mod 8$, then $|A|=(2^{\alpha -3})^{k-1}$, which establishes the formula.
\end{proof}
\begin{rmk}
Let $n=p_{1}^{\alpha_{1}}\ldots p_{t}^{\alpha_{t}}$. Then by Lemma \ref{product}, we conclude that $$\mathscr{S}_{sq} (\mathbb{Z}_{n},c,k)=\prod_{i=1}^{t} \mathscr{S}_{sq}(\mathbb{Z}_{p_{i}^{\alpha_{i}}},c,k),$$
which can be computed easily by Theorems \ref{4k1}, \ref{np4k3} and \ref{n2alpha}.
\end{rmk}

\bigskip
{\bf Acknowledgments.}\\ The author deeply thanks Dariush Kiani for encouragement. The author also thanks the referee for careful reading and useful comments.
{}

\end{document}